\documentclass[a4paper,11pt,reqno]{amsart}
\usepackage{amsmath}
\usepackage{amsthm,enumerate}

\usepackage{graphicx}
\usepackage{amssymb}

\usepackage[latin1]{inputenc}
\usepackage{appendix}

\usepackage[italian,english]{babel}

\usepackage{fancyhdr}

\usepackage{calc}
\usepackage{url}

\usepackage[text={6in,8.6in},centering]{geometry}

\usepackage{srcltx}

\usepackage[T1]{fontenc}

\usepackage[usenames,dvipsnames]{color}

\theoremstyle{plain}
\newtheorem{thm}{Theorem}[section]
\theoremstyle{plain}
\newtheorem{lem}[thm]{Lemma}
\newtheorem{prop}[thm]{Proposition}
\newtheorem{cor}[thm]{Corollary}

\theoremstyle{definition}

\newcommand{\rn}{\mathbb{R}^{N}}

\newcommand{\hn}{\mathbb{H}^{N}}

\newcommand{\authorfootnotes}{\renewcommand\thefootnote{\@fnsymbol\c@footnote}}%

\numberwithin{equation}{section} \allowdisplaybreaks

\begin{document}
        \title[]{Improved multipolar Poincar\'e-Hardy inequalities on Cartan-Hadamard manifolds}

\date{}

\author[Elvise BERCHIO]{Elvise BERCHIO}
\address{\hbox{\parbox{5.7in}{\medskip\noindent{Dipartimento di Scienze Matematiche, \\
Politecnico di Torino,\\
        Corso Duca degli Abruzzi 24, 10129 Torino, Italy. \\[3pt]
        \em{E-mail address: }{\tt elvise.berchio@polito.it}}}}}
\author[Debdip GANGULY]{Debdip GANGULY}
\address{\hbox{\parbox{5.7in}{\medskip\noindent{Department of Mathematics,\\
 Indian Institute of Science Education and Research,\\
 Dr. Homi Bhabha Road, Pashan\\
        Pune 411008, India. \\[3pt]
        \em{E-mail address: }{\tt debdipmath@gmail.com}}}}}
\author[Gabriele GRILLO]{Gabriele GRILLO}
\address{\hbox{\parbox{5.7in}{\medskip\noindent{Dipartimento di Matematica,\\
Politecnico di Milano,\\
   Piazza Leonardo da Vinci 32, 20133 Milano, Italy. \\[3pt]
        \em{E-mail addresses: }{\tt
          gabriele.grillo@polimi.it}}}}}


\keywords{Hyperbolic space, multipolar Hardy inequality, Poincar\'e inequality}


\begin{abstract}

We prove a family of improved multipolar Poincar\'e-Hardy inequalities on Cartan-Hadamard manifolds. For suitable configurations of poles, these inequalities yield an improved multipolar Hardy inequality and an improved multipolar Poincar\'e inequality such that the critical unipolar singular mass is reached at any pole.
\end{abstract}

\maketitle

 \section{Introduction}

This paper aims at proving new multipolar Hardy inequalities on negatively curved manifolds. To introduce the subject, let us recall the simplest form of the unipolar Hardy inequality on Riemannian manifolds, which is due to Carron \cite{Carron}. If $(\mathcal{M},g)$ is an $N\geq 3$ dimensional Cartan-Hadamard manifold, ${\rm d}(., .)$ is the geodesic distance and $x_0 \in \mathcal{M}$, the following inequality holds
\begin{equation}\label{hardy}
 \int_{\mathcal{M}} |\nabla_{g} u|^2 \ {\rm d}v_g \geq \frac{(N -2)^2}{4} \int_{\mathcal{M}} \frac{u^2}{{\rm d}(x, x_0)^2} \ {\rm d}v_g \quad \forall \ u \in C_{c}^{\infty}(\mathcal{M})\,.
 \end{equation}
 Here $\nabla_{g}$ is the Riemannian gradient, ${\rm d}v_g$ the Riemannian volume.

When $\mathcal{M}=\rn$, the Euclidean space, or $\mathcal{M}=\hn$, the hyperbolic space, the constant $ \frac{(N -2)^2}{4}$ is optimal and never attained. After the seminal work \cite{Carron} there have been several attempt to improve \eqref{hardy} in terms of adding remainder terms, see for instance \cite{ AK, Kombe1, Kombe2}.

 \medskip
 Motivated by applications to various fields of research, several authors have turned their attention to the multipolar version of \eqref{hardy}. In $\rn$ a relevant contribution in this direction is the one of Bossi-Dolbeault-Esteban in \cite{BDE}. Letting $y_1, \ldots y_M$ be $M\geq 2$ distinct points in $\mathbb{R}^N$ and  $d= {\rm min}_{i \neq j} \frac{|y_i - y_j|}{2}$, they showed that \begin{equation}\label{m_euc_hardy}
\int_{\mathbb{R}^N} |\nabla u|^2 \, {\rm d}x + \frac{4\pi^2+(M+1)(N-2)^2}{4d^2}
\int_{\mathbb{R}^N} u^2 \, {\rm d}x \geq  \frac{(N -2)^2}{4} \sum_{i = 1}^M
\int_{\mathbb{R}^N} \dfrac{u^2}{|x - y_i|^2}{\rm d}x  ,
\end{equation}
for any $u \in H^{1}(\rn) $. In few words, their result shows that the unipolar critical mass, $\frac{(N -2)^2}{4}\frac{1}{|x-y_i|^2}$, is reached at any singular pole $y_i$, modulo adding a lower order $L^2$-term on the left hand side. Trying to remove the above
$L^2$-correction term, the authors in \cite{BDE} also proved the following inequality:

{\small \begin{equation}\label{m_euc_hardy_2}
\int_{\rn} |\nabla u|^2 \, {\rm d}x  \geq \frac{(N -2)^2}{4M} \sum_{i = 1}^{M} \int_{\rn} \frac{u^2}{|x - y_i|^2} \, {\rm d}x \, +\,
 \frac{(N - 2)^2}{4 M^2} \sum_{1 \leq i< j \leq M} \int_{\rn} \frac{|y_i - y_j|^2}{|x - y_i|^2 |x - y_j|^2} u^2 \, {\rm d}x\,,
\end{equation}}
for any $u \in H^{1}(\rn) $. Here the contribution at each singularity $y_i$ is
 $$\dfrac{(N - 2)^2}{4} \dfrac{2M-1}{M^2} \dfrac{1}{|x - y_i|^2}\,.$$
  Hence, a constant strictly less than $\frac{(N-2)^2}{4}$ has to be taken in order to remove the positive correction on the left hand side of \eqref{m_euc_hardy}. In fact, a remarkable result of Felli, Marchini and Terracini in \cite{felli} says that there exists no configuration of the poles such that a multipolar Hardy inequality holds with any of the Hardy terms contributing with the best unipolar constant $(N-2)^2/4$, or even with the sum of such constants being equal to $(N-2)^2/4$.

 The optimality issues were not addresses in \cite{BDE} while in \cite{CZ} the following \emph{optimal}  inequality was proved

 \begin{equation}\label{m_euc_hardy_3}
\int_{\rn} |\nabla u|^2 \, {\rm d}x  \geq
 \frac{(N - 2)^2}{ M^2} \sum_{1 \leq i < j \leq M}  \int_{\mathbb{R}^N} \frac{|y_i - y_j|^2}{|x - y_i|^2 |x - y_j|^2} u^2 \,dx
\end{equation}
for any  $u \in H^{1}(\rn) $. Since the potential in \eqref{m_euc_hardy_3} behaves asymptotically near $y_i$ as
$$ \frac{(N-2)^2}{4}
 \frac{4M - 4}{M^2} \frac{1}{|x - y_i|^2}\,,
 $$
for $M \geq 2$ it provides a larger weight near poles than that in \eqref{m_euc_hardy_2}. On the other hand, as $x\rightarrow +\infty$ the decay of the potential in \eqref{m_euc_hardy_3} is faster than that in \eqref{m_euc_hardy_2}, hence the two weights are not globally comparable. An inequality similar to \eqref{m_euc_hardy_2} and \eqref{m_euc_hardy_3} is also proved in \cite{ADI} using different techniques.
It is worth recalling a further multipolar Hardy inequality from \cite{DFP} :
 {\small \begin{equation}\label{m_euc_hardy_4}
\int_{\rn} |\nabla u|^2 \, {\rm d}x  \geq \frac{(N -2)^2}{(M+1)^2} \left[\sum_{i = 1}^{M} \int_{\rn} \frac{u^2}{|x - y_i|^2} \, {\rm d}x \, +\,
  \sum_{1 \leq i< j \leq M} \int_{\rn} \frac{|y_i - y_j|^2}{|x - y_i|^2 |x - y_j|^2} u^2 \, {\rm d}x \right]
\end{equation}}
for any  $u \in H^{1}(\rn) $. The potential in \eqref{m_euc_hardy_4} is smaller near poles than those listed above, indeed it behaves like
 $$ \dfrac{(N - 2)^2}{4} \,\frac{4M}{(M + 1)^2}\,\frac{1}{|x- y_i|^2}\,, $$
nevertheless the resulting operator turns out to be \emph{critical}, namely inequality \eqref{m_euc_hardy_4} cannot be further improved.  It should be said that in \cite{DFP} the authors proved also the criticality of the operator associated to \eqref{m_euc_hardy_3}.

 \medskip

In view of \eqref{hardy}, it is quite natural to generalize the above results to the non Euclidean setting. 
To our best knowledge, the first result in this direction is that given in \cite{FFK} where the authors, following the super solutions approach of \cite{CZ} to get \eqref{m_euc_hardy_3}, for any $(\mathcal{M},g)$ $N$-dimensional complete Riemannian manifold proved that
    \begin{align}\label{hardy_multi_crista1}
 \int_{\mathcal{M}} |\nabla_{g} u|^2 \, {\rm d}v_{g}
 &  \geq    \frac{(N-2)^2}{M^2}  \sum_{1 \leq i< j \leq M} \int_{\mathcal{M}} \left| \frac{\nabla_g ({\rm d}(x, y_i)) }{{\rm d}(x, y_i)}-\frac{\nabla_g ({\rm d}(x, y_j)) }{{\rm d}(x, y_j)}\right|^2 u^2 \, {\rm d}v_{g}\\ \notag
 &+ \frac{(N-2)(N-1)}{M}  \sum_{k = 1}^M  \int_{\mathcal{M}} \frac{{\rm d}^2(x, y_k) \Delta_g ({\rm d}^2(x, y_k))-(N-1)}{{\rm d}^2(x, y_k)}\,u^2 \, {\rm d}v_{g}
 \end{align}
 for all $u \in C_{c}^{\infty} (\mathcal{M})$. Here $\Delta_g$ is the Laplace-Beltrami operator.

 If $\mathcal{M}$ has asymptotically non-negative Ricci curvature, the operator associated to \eqref{hardy_multi_crista1} is also proved to be critical in the above mentioned sense. Furthermore, on Cartan-Hadamard manifolds the last term of \eqref{hardy_multi_crista1} is nonnegative, while it vanishes when $\mathcal{M}=\rn$ and one recovers exactly inequality \eqref{m_euc_hardy_3}.  In any case, the total mass at each pole never reaches the optimal unipolar mass of \eqref{hardy}. Motivated by this remark, in the present paper we combine the localizing approach of \cite{BDE} with some optimal Poincar\'e-Hardy unipolar inequalities from \cite{BGGP}, and we investigate whether the contribution of the curvature helps in removing the l.h.s. correction term of \eqref{m_euc_hardy}; namely in getting a multipolar Hardy inequality with optimal unipolar critical mass reached at each pole.

 The proposed goal is reached on Cartan-Hadamard manifolds having sectional curvatures bounded from above by a strictly negative constant, provided the distance between poles satisfies a suitable lower bound. Remainder terms for the multipolar Hardy inequality are also provided. See Corollary \ref{hardy_thm_manif} for the precise statement and the discussion just below for a comparison between our inequality and inequality \eqref{hardy_multi_crista1}. The proof of our multipolar Hardy inequality follows as a corollary of a more general family of improved multipolar Poincar\'e-Hardy inequalities, see Theorem \ref{multip_general}, having its own interest and from which we also derive a multipolar Hardy improvement for the (non optimal) Poincar\'e inequality such that the critical unipolar singular mass is reached at any pole, see Corollary \ref{hardy_thm_lamc}. \medskip

The paper is organized as follows. In Section \ref{hn}, for the sake of clarity, we state our results on the prototype Cartan-Hadamard manifold, namely on the hyperbolic space $\hn$. In Section \ref{hnproof} we prove all the statements in the hyperbolic setting. Finally, in Section \ref{general} we state our results on general Cartan-Hadamard manifolds and in Section \ref{super_solution_sec} we discuss the criticality issue for a further multipolar Hardy inequality in $\hn$ that we derive from the general results of \cite{DFP}.

\section{Main results in the hyperbolic setting.}\label{hn}
We start by recalling a unipolar Poincar\'e-Hardy inequality from \cite{BGGP} that will play a key role in our proofs.  Let $ \lambda_1(\hn)$ denote the bottom of the $L^2$ spectrum of $-\Delta_{\hn}$. For all $ \lambda \leq  \lambda_{1}(\hn)=\left(\frac{N-1}{2} \right)^2$, consider the potential
  \begin{equation}\label{potential}
V_{\lambda,x_0}(x):= H_N(\lambda)\, \frac{1}{{\rm d}^2(x, x_0)}+ C_N(\lambda)  \frac{1}{\sinh^2 ({\rm d}(x, x_0))}+ D_N(\lambda) g({\rm d}(x, x_0))\,,
\end{equation}
where the function $g$ is defined by
\begin{equation*}
g(r) := \frac{r \coth r - 1}{r^2} >0 \quad \text{for }r>0\,.
\end{equation*}
Hence, $g$ is strictly decreasing and satisfies
$$g(r) \sim \frac{1}{3} \ \mbox{ as } \  r \rightarrow 0^+ \quad \mbox{ and }  \quad g(r) \sim \frac{1}{r} \  \mbox{ as }  \  r \rightarrow +\infty\,.$$
Furthermore, the constants $H_N,C_N,D_N$ are given explicitly by
\begin{equation}\label{constants}
\begin{aligned}
 H_N(\lambda):=&\frac{(\gamma_{N}(\lambda)+1)^2}{4}\,, \quad D_N(\lambda):=\frac{\gamma_{N}(\lambda)(\gamma_{N}(\lambda)+1)}{2} \\
 C_N(\lambda):=& \frac{(N-1+\gamma_{N}(\lambda))(N-3-\gamma_{N}(\lambda))}{4}\,,\,
\end{aligned}
\end{equation}
 with $\gamma_{N}(\lambda):=\sqrt{(N-1)^2-4\lambda}$. It is worth noting that the constant $H_N(\lambda)$ cannot exceed the Hardy constant and its maximum value is achieved for $\gamma_{N}(N-2)=N-3$, namely for $\lambda=N-2$. We may finally state: \begin{prop}\cite[Theorem~2.1]{BGGP}\label{improved-hardy}
Let $N \geq 3$ and $x_0\in \hn$. For all $N-2 \leq \lambda \leq  \left(\frac{N-1}{2} \right)^2$ and all $u \in C_{c}^{\infty} (\hn)$ there holds
\begin{equation}\label{unip}
\int_{\hn}|\nabla_{\hn} u|^2\, \emph{d}v_{\hn} -\lambda \int_{\hn}  u^2\, \emph{d}v_{\hn}\\
  \geq  \int_{\hn} V_{\lambda,x_0}(x) \,u^2 \ \emph{d}v_{\hn} \,,
\end{equation}
where $V_{\lambda,x_0}$ is as defined in \eqref{potential}. Besides, the operator $ -\Delta_{\hn} -\lambda -V_{\lambda,x_0}$ is critical in $\hn$ in the sense that the inequality
$$
\int_{\hn} |\nabla_{\hn} u|^2 \ \emph{d}v_{\hn}-\lambda \int_{\hn}  u^2\, \emph{d}v_{\hn} \geq \int_{\hn} V(x) u^2 \ \emph{d}v_{\hn}
 \quad\forall u \in C_{c}^{\infty} (\hn)
$$
is not valid for all $u \in C_{c}^{\infty} (\hn)$ given any $V \gneqq V_{\lambda} $.
\end{prop}
We note that Proposition \ref{improved-hardy} holds with the same statement also for $\lambda<N-2$ but in that case, $\gamma_{N}(N-2)>N-3$ and the constant $C_N(\lambda)$ becomes negative. Hence, even if $V_{\lambda,x_0}$ is still a positive function, the singularity at $x_0$ splits into a positive and a negative part. For this reason, in what follows, in order to avoid too many distinctions, we will always assume $\lambda \geq N-2$. We refer the interested reader to \cite{BG,BGG, BAGG,BGGP, Ngo, Ngu} for more details and for related Poincar\'e-Hardy inequalities in the higher order case or in the $L^p$-setting.
\medskip

The main idea behind the proof of Proposition \ref{improved-hardy} is to find a suitable radial supersolution. A supersolution technique has been used in the multipolar case in \cite{DFP}, in the Euclidean case, and in \cite{FFK} in the Riemannian case. Some results along this line will also be proved here in Section \ref{super_solution_sec}, but here, in order to get locally stronger inequalities, we adapt instead the approach of \cite{BDE}. Namely, we first use the so-called
``IMS'' truncation method (see \cite{JDM} ) to localize to Hardy potentials with
one singularity, then we refine the estimates by means of inequality \eqref{unip} which allows us to exploit some positive terms on the r.h.s. not usable in the flat case.

Let $M \geq 2$ and  $y_1, \ldots y_M \in \hn$, we consider the multipolar analogous to \eqref{potential}, namely
 \begin{align}\label{non_sym_pot}
V_{\lambda,y_1, \ldots y_M }(x) : = \sum_{k = 1}^M  V_{\lambda,y_k}(x)\,,
\end{align}
where the functions $V_{\lambda,y_k}(x)$ are as in \eqref{potential}. Now we set
 \begin{equation}\label{d}
d  :=  \min_{1 \leq j \neq k \leq M} \frac{{\rm d}(y_j, y_k)}{2}
 \end{equation}
 and we state our main theorem

\begin{thm}\label{main_thm}
Let $N \geq 3$ and $M \geq 2$. Furthermore, let $y_1, \ldots y_M \in \hn$ and let $d>0$ be defined in \eqref{d}. For all $ N-2\leq \lambda \leq  \left(\frac{N-1}{2} \right)^2$, the following inequality holds

 \begin{equation}\label{eq_main_thm}
 \int_{\hn} |\nabla_{\hn} u|^2 \, {\rm d}v_{\hn}  +  \left( \frac{\pi^2}{d^2}+ (M+1)K_{N}(\lambda,d) -\lambda\right)\int_{\hn} u^2 {\rm d}v_{\hn}
  \geq     \int_{\hn} V_{\lambda,y_1, \ldots y_M }(x)\, u^2 {\rm d}v_{\hn}
 \end{equation}
 for all $u \in C_{c}^{\infty} (\hn)$, where $V_{\lambda,y_1, \ldots y_M }$ is as given in \eqref{non_sym_pot} while
 \begin{equation}\label{Vd}
  K_{N}(\lambda,d):= H_N(\lambda)\, \frac{1}{d^2}+ C_N(\lambda)  \frac{1}{\sinh^2 (d/2)}+ D_N(\lambda)  g(d/2) \,,
 \end{equation}
with the constants $H_N,C_N,D_N$ as in \eqref{constants} and the function $g$ as in \eqref{potential}.
 \end{thm}

\medskip
 A few special cases have to be singled out. In fact, for $\lambda= N-2$ one has $H_N(N-2)=\frac{(N-2)^2}{4}$, $C_N(N-2)=0$, $D_N(N-2)=\frac{(N-3)(N-2)}{2}$, and Theorem \ref{main_thm} provides an interesting inequality that we state separately here below. Let $\bar d=\bar d(M,N)>0$ be defined implicitly as follows
  \begin{equation}\label{overd}
  \frac{\pi^2 }{\bar d^2}+(M+1)\left(\frac{(N-2)^2}{4} \frac{1}{\bar d^2}+ \frac{(N-3)(N-2)}{2} g(\bar d/2)\right)=N-2\,.
   \end{equation}
 Then we have the following result:
\begin{cor}\label{hardy_thm}
Let $N \geq 3$ and $M \geq 2$. Furthermore, let $y_1, \ldots y_M \in \hn$, $d>0$ as defined in \eqref{d} and $\bar d=\bar d(M,N)>0$ as defined in \eqref{overd}.  If $d\geq \bar d $, the following inequality holds

 \begin{equation}\label{hardy_multi}
  \begin{aligned}
 \int_{\hn} |\nabla_{\hn} u|^2 \, {\rm d}v_{\hn}
 &  \geq    \frac{(N-2)^2}{4}  \sum_{k = 1}^M  \int_{\hn}  \frac{u^2}{{\rm d^2}(x, y_k)} \, {\rm d}v_{\hn}\\
 &+ \frac{(N-2)(N-3)}{2}  \sum_{k = 1}^M  \int_{\hn} g({\rm d}(x, y_k))\,u^2 \, {\rm d}v_{\hn}
 \end{aligned}
 \end{equation}
 for all $u \in C_{c}^{\infty} (\hn)$ with the function $g$ as in \eqref{potential}.
 \end{cor}
The interest of inequality \eqref{hardy_multi} is clearly related to the fact that one recovers the unipolar Hardy potential with best constant at each pole, and further positive remainder term on the right hand side, the resulting inequality being locally stronger than \eqref{hardy_multi_crista1}. This turns out to be true under the assumption $d\geq \bar d $, namely the poles cannot be too close to each other. With a similar assumption, exploiting inequality \eqref{eq_main_thm} with $\lambda=\left(\frac{N-1}{2} \right)^2$, one can also derive a multipolar improvement of a Poincar\'e inequality with constant arbitrarily close to the optimal one. More precisely, for $N-2\leq  \lambda < \left(\frac{N-1}{2} \right)^2$, let $\bar d_{\lambda}=\bar d_{\lambda}(M,N)>0$ be defined implicitly as follows
  \begin{equation}\label{overdlambda}
  \frac{\pi^2 }{\bar d_{\lambda} }+(M+1)\left(\frac{1}{4} \frac{1}{\bar d_{\lambda}^2}+\frac{(N-1)(N-3)}{4}\  \frac{1}{\sinh^2 (\bar d_{\lambda}/2)}\right)=\left(\frac{N-1}{2} \right)^2-\lambda\,.
   \end{equation}
 Then the following result holds:
\begin{cor}\label{hardy_thm_lam}
Let $N \geq 3$ and $M \geq 2$. Furthermore, let $y_1, \ldots y_M \in \hn$, $d>0$ as defined in \eqref{d} and $\bar d_{\lambda}=\bar d_{\lambda}(M,N)>0$ as defined in \eqref{overdlambda}. Let $N-2\leq \lambda < \left(\frac{N-1}{2} \right)^2$, if $d\geq \bar d_{\lambda} $, the following inequality holds

  \begin{equation}\label{poin_multi}\begin{aligned}
 \int_{\hn} |\nabla_{\hn} u|^2 \, {\rm d}v_{\hn}
 &  \geq  \lambda   \int_{\hn}  u^2 \, {\rm d}v_{\hn} +   \frac{1}{4}  \sum_{k = 1}^M  \int_{\hn}  \frac{u^2}{{\rm d^2}(x, y_k)} \, {\rm d}v_{\hn}\\
 &+ \frac{(N-1)(N-3)}{4}  \sum_{k = 1}^M  \int_{\hn}  \frac{u^2}{ \sinh^2 ({\rm d}(x, y_k))}\, \, {\rm d}v_{\hn}
 \end{aligned} \end{equation}
 for all $u \in C_{c}^{\infty} (\hn)$.
 \end{cor}
We note that the Hardy improvement of \eqref{poin_multi} is locally optimal for all $\lambda$ in the sense that
$$\frac{1}{4} \frac{1}{  {\rm d^2}(x, y_k)}+ \frac{(N-1)(N-3)}{4 } \frac{1}{ \sinh^2 ({\rm d}(x, y_k))}\sim \frac{(N-2)^2}{4}  \frac{1}{  {\rm d^2}(x, y_k)}\quad \text{as } x \rightarrow y_k\,,$$ namely the unipolar critical mass is reached at any pole. It is also to be noted that the constant 1/4 in the unipolar inequality
\[
 \int_{\hn} |\nabla_{\hn} u|^2 \, {\rm d}v_{\hn}
  \geq  \frac{(N-1)^2}{4}   \int_{\hn}  u^2 \, {\rm d}v_{\hn} +   \frac{1}{4}    \int_{\hn}  \frac{u^2}{{\rm d^2}(x, x_0)} \, {\rm d}v_{\hn}
\]
is known to be sharp by the results of \cite{BGG}.

\section{Proof of Theorem~\ref{main_thm}}\label{hnproof}

 \subsection{Technical lemmas}
 We adapt the approach of \cite[Section 1]{BDE} to our framework. Consider a partition of unity in $\hn$, namely a finite set $\{ J_k \}_{k = 1}^{M+1}$ of
  real valued functions $J_k \in W^{1, \infty}(\hn)$ such that $
 \sum_{k = 1}^{M + 1} J_k^2 = 1$ a.e. in $\hn$. Hence, $J_{M + 1} :=
\sqrt{1 - \sum_{k = 1}^{M} J_k^2}$ and, by writing Riemannian gradient in local coordinates, i.e., $(\nabla_{\hn})^j = g^{i j} \partial_i$, one has $\sum_{k = 1}^{M + 1} J_k (\nabla_{\hn})^jJ_k = 0$ for every $j=1,...,N$. Furthermore, if one requires that
\begin{equation}\label{parti_assump}
{\rm Int} ({\rm supp}(J_k)) \cap {\rm Int} ({\rm supp}(J_l))= \emptyset \quad \mbox{for any} \ k, l = 1, \ldots, M, k \neq l\,,
\end{equation}
 it follows that
 \begin{equation}\label{alternate_sum}
 \sum_{k = 1}^{M + 1} |\nabla_{\hn} J_k|^2 = \sum_{k = 1}^{M} \frac{|\nabla_{\hn} J_k|^2}{1 - J_k^2} \quad \text{a.e. in } \hn.
 \end{equation}
 First we state

\begin{lem}\label{lem2}
Let $\{ J_k \}_{k = 1}^{M + 1}$ be a partition of unity satisfying condition \eqref{parti_assump}. For any $u \in H^{1}(\hn)$ and potential $V \in L^1_{loc} (\hn),$ there holds
\begin{align*}
\langle Lu, u \rangle   := \int_{\hn} |\nabla_{\hn} u|^2 \, {\rm d}v_{\hn} -  \int_{\hn} V(x) u^2 \, {\rm d}v_{\hn} = \sum_{k = 1}^{M} \langle L(J_k u), (J_k u) \rangle + R_M\,,
\end{align*}
with
\begin{align*}
R_M & = \int_{\hn}  |\nabla_{\hn} (J_{M + 1} u)|^2 -  \int_{\hn \setminus\Omega} V(x)\, u^2 \, {\rm d}v_{\hn}\\
& - \sum_{k = 1}^{M} \int_{\Omega_k} \left( \frac{|\nabla_{\hn} J_k|^2}{1 - J_k^2(x)} +  (1 - J_k^2(x)) V(x) \right) \, u^2 \, {\rm d}v_{\hn}\,,
\end{align*}
where  $\Omega_k:={\rm Int}({\rm supp}(J_k))$ and $\Omega:= \cup_{k = 1}^{M} {\rm supp}(J_k)$.
\end{lem}
\begin{proof}
Using the fact that
$$
\left|\nabla_{\hn} (J_k u) \right|^2 = \left|u (\nabla_{\hn} J_k) + J_k (\nabla_{\hn} u) \right|^2 =$$
$$ u^2 |\nabla_{\hn} J_k|^2 + J_k^2 |\nabla_{\hn} u|^2 +
 (J_k \nabla_{\hn} J_k\,,\nabla_{\hn}(u^2))_{\hn}
$$

and exploiting the properties of the $J_k$ listed above, one has that
\begin{align*}
   \langle Lu, u \rangle  &= \sum_{k = 1}^{M + 1} \int_{\hn}
   \langle L(J_k u), (J_k u) \rangle  - \int_{\hn} \left( \sum_{k = 1}^{M + 1} |\nabla_{\hn} J_k|^2 \right) u^2 \,  {\rm d}v_{\hn}\\
   & = \sum_{k = 1}^{M} \langle L(J_k u), (J_k u) \rangle + \int_{\hn} \left( |\nabla_{\hn} (J_{M + 1} u)|^2 -  \, V(x) (J_{M + 1} u)^2 \right) \, {\rm d}v_{\hn} \\
   &  - \int_{\hn} \left( \sum_{k = 1}^{M + 1} |\nabla_{\hn} J_k|^2 \right) u^2 \,  {\rm d}v_{\hn} =: \sum_{k = 1}^{M} \langle L(J_k u), (J_k u) \rangle + R_M\,.
\end{align*}
Note that, by \eqref{alternate_sum}, $R_M$ can rewritten as

\begin{align*}
R_M  &  =  \int_{\hn}  |\nabla_{\hn} (J_{M + 1} u)|^2  -  \int_{\hn} V(x) \left(1 - \sum_{k = 1}^{M} J_k^2(x) \right) u^2 \, {\rm d}v_{\hn}\\
   &   -   \sum_{k = 1}^{M } \int_{\hn}  \frac{|\nabla_{\hn} J_k|^2}{1 - J_k^2} \, u^2 \,  {\rm d}v_{\hn}
\end{align*}
and the statement of Lemma \ref{lem2} follows by recalling \eqref{parti_assump} and writing $\hn=(\hn \setminus \cup_{k = 1}^{M} {\rm supp}(J_k) )\cup (\cup_{k = 1}^{M} {\rm supp}(J_k))$.
\end{proof}

 \medskip

Next we denote by $B(x_0, r)$ the geodesic ball of center $x_0\in \hn$ and radius $r>0$ and we prove:

\begin{lem}\label{key_lemma}
 Let $y_1,y_2\in \hn$ with $ 0< {\rm d}(y_1, y_2)=:2d$. There is a partition of unity $\{ J_k \}_{k = 1}^{3}$ with $J_i \equiv 1$
on $B(y_i, \frac{d}{2})$ and $J_i \equiv 0$ on $B(y_i, d)^c,$ for $i=1,2$. Furthermore, for every $N-2 \leq \lambda\leq (\frac{N-1}{2})^2$ and for a.e. $x\in B(y_1, d) \cup B(y_2, d)$,
 there holds
 $$\sum_{k = 1}^{2} \frac{|\nabla_{\hn} J_k(x)|^2}{1 - J_k^2(x)} +  (1-J_1^2(x)-J_2^2(x))\, V_{\lambda,y_1, y_2 }(x) \leq  \frac{\pi^2}{d^2}+2K_{N}(\lambda,d)\,,
 $$
where $V_{\lambda,y_1, y_2 }(x) $ is as defined in \eqref{non_sym_pot} and $K_{N}(\lambda,d)$ is as given in \eqref{Vd}.
\end{lem}

\begin{proof}
Let $J:[0,+\infty) \rightarrow [0,1]$ be a continuous map defined as follows
\[J(t)=\left\{\begin{array}{ll}
1 & \text{if $0\leq t \leq \frac{1}{2}$\,}
\\
{\rm sin} (\pi t) & \text{if $\frac{1}{2} \leq t \leq 1$\,}
\\
0 & \text{if $ t \geq 1$\,.}
\end{array}
\right.
\]
For all $x\in \hn,$  we set $J_1 (x) : = J \left( \frac{{ \rm d}(x, y_1)}{d} \right),$ and $J_2 (x) : = J \left( \frac{{ \rm d}(x, y_2)}{d} \right)$,
and $J_3(x) := \sqrt{1 - J_1^2(x) - J^2_2(x)}$.  It is readily seen that $J_i \in W^{1, \infty}(\hn)$ for $i=1,...,3$. Then, we consider
\begin{align*}
\sup_{x\in B(y_1, d)}  \left[ \frac{|\nabla_{\hn} J_1(x)|^2}{1 - J_1^2(x)} +  (1 - J_1^2(x))  \, \left( V_{\lambda,y_1}(x)+V_{\lambda,y_2}(x)\right) \right].
\end{align*}
Since for $y_1\in \hn$ fixed, we have $|\nabla_{\hn} { \rm d}(x, y_1)|^2=1$, for a.e. $x\in B(y_1, d)$ we infer
\begin{align*}
|\nabla_{\hn} J_1(x)|^2 = \frac{1}{d^2} \left| J^{\prime} \left( \frac{ { \rm d}(x, y_1)}{ d} \right) \right|^2 |\nabla_{\hn} { \rm d}(x, y_1)|^2=\frac{1}{d^2} \left| J^{\prime} \left(t\right) \right|^2\ \quad \text{with }  t :=\frac{ { \rm d}(x, y_1)}{ d}\in [0,1)\,
\end{align*}
and, using the fact that
\begin{equation*}
{\rm d}(x, y_2)  \geq {\rm d}(y_1, y_2) -{\rm d}(x, y_1) = (2-t)d\,,
\end{equation*}
 we obtain
\begin{align*}
 V_{\lambda,y_1}(x)+V_{\lambda,y_2}(x)& \leq \frac{H_N(\lambda)}{d^2}\left[\frac{1}{t^2} + \frac{1}{(2-t)^2}\right]+C_N(\lambda)\left[\frac{1}{\sinh^2(td)}
 \,+ \,\frac{1}{\sinh^2((2-t)d)}\right]\\
& +D_N(\lambda)\left[ g(td)+g((2-t)d) \right]=:W_{\lambda}(t,d)\,.
\end{align*}
From the above computations we deduce
\begin{align*}
& \sup_{x\in B(y_1, d)}   \left[ \frac{|\nabla_{\hn} J_1(x)|^2}{1 - J_1^2(x)} +  (1 - J_1^2(x))  \,\left( V_{\lambda,y_1}(x)+V_{\lambda,y_2}(x)\right) \right] \\
& \leq \sup_{t \in (\frac{1}{2}, 1)} \left[  \frac{(J^{\prime}(t))^2}{d^2(1 - J^2(t))} + (1 - J^2(t))\,W_{\lambda}(t,d) \right]\\
&\leq \frac{\pi^2 }{d^2}+\max_{t \in [\frac{1}{2}, 1]} \left[ \frac{H_N(\lambda)}{d^2}\left[\frac{1}{t^2} + \frac{1}{(2-t)^2}\right]\, \cos^{2}(\pi t) \right] \\
&+C_N(\lambda)\frac{2}{\sinh^2(d/2)} + 2D_N(\lambda)\,g(d/2)=\frac{\pi^2}{d^2}+2K_{N}(\lambda,d)\,,
   \end{align*}
 where $K_{N}(\lambda,d)$ is given in \eqref{Vd}. Similar computations, inverting the roles of $y_1$ and $y_2$, give
 \begin{align*}
 \sup_{x\in B(y_2, d)}   \left[ \frac{|\nabla_{\hn} J_2(x)|^2}{1 - J_2^2(x)} +  (1 - J_2^2(x))  \,\left( V_{\lambda,y_1}(x)+V_{\lambda,y_2}(x)\right) \right]\leq \frac{\pi^2}{d^2}+2K_{N}(\lambda,d)\,.
\end{align*}
This completes the proof.
\end{proof}


  \subsection{Hardy inequality with two singularities}
 This section is devoted to prove Theorem~\ref{main_thm} in the case of two singularities. More precisely, we shall prove the following result.
 \begin{prop}\label{sym2_thm}

Let $N \geq 3,$ Let $y_1,y_2\in \hn$ with $ 0< {\rm d}(y_1, y_2)=:2d$. For all $ N-2\leq \lambda \leq  \left(\frac{N-1}{2} \right)^2$
 the following inequality
 \begin{equation*}
\int_{\hn} |\nabla_{\hn} u|^2 \, {\rm d}v_{\hn}  +\left( \frac{\pi^2}{d^2}+3K_{N}(\lambda,d) -\lambda \right)\int_{\hn } u^2\, {\rm d}v_{\hn} \geq  \int_{\hn}  V_{\lambda,y_1, y_2 }(x)\, u^2 \, {\rm d}v_{\hn}
  \end{equation*}
holds for all $u \in C_{c}^{\infty} (\hn)$, where  $K_{N}(\lambda,d)$ is given in \eqref{Vd} and $V_{\lambda,y_1, y_2 }$ is given in \eqref{non_sym_pot}.
\end{prop}

 \begin{proof}
Consider the partition of unity given by Lemma \ref{key_lemma}, from Lemma \ref{lem2} with $V=V_{\lambda,y_1, y_2 }$ we have
 \begin{align*}
 \langle Lu, u \rangle  & := \int_{\hn} |\nabla_{\hn} u|^2 \, {\rm d}v_{\hn} -  \int_{\hn}  V_{\lambda,y_1, y_2 }(x)\, u^2 \, {\rm d}v_{\hn}= \sum_{k = 1}^{2} \langle L(J_k u), (J_k u) \rangle + R_2\,,
\end{align*}
where
 \begin{align*}
R_2 & = \int_{\hn}  |\nabla_{\hn} (J_{3} u)|^2 \, {\rm d}v_{\hn} \, -  \int_{\hn \setminus  \Omega}  V_{\lambda,y_1, y_2 }(x)\, u^2 \, {\rm d}v_{\hn} \\
& - \sum_{k = 1}^{2} \int_{B(y_k, d) } \left( \frac{|\nabla_{\hn} J_k|^2}{1 - J_k^2(x)} +  (1 - J_k^2(x))  V_{\lambda,y_1, y_2 }(x) \right) \, u^2 \, {\rm d}v_{\hn}\,,
\end{align*}
and $\Omega=\overline{B(y_1, d)} \cup \overline{B(y_2, d)}$.
In particular if $x\in B(y_1,d)$, since ${\rm d}(x, y_2)> d$ and recalling \eqref{potential} and \eqref{non_sym_pot}, we have
$$V_{\lambda,y_1, y_2 }(x)\leq V_{\lambda,y_1}(x)+V_{\lambda}(d)\,,$$
where
 \begin{equation} \label{Vdd}
V_{\lambda}(d):= \left(H_N(\lambda)\, \frac{1}{ d^2}+ C_N(\lambda)  \frac{1}{\sinh^2 (d)}+ D_N(\lambda) g(d)\right)\,.
  \end{equation}
A similar argument works if $x\in B(y_2,d)$.  Using this fact and exploiting in each $B(y_k, d)$ inequality \eqref{unip} with $x_0=y_k$, for all $N-2 \leq \lambda\leq \left(\frac{N-1}{2}\right)^2$, we get
\begin{equation*}
\sum_{k= 1}^2 \langle L(J_k u), (J_k u) \rangle \geq \lambda \int_{\Omega} ((J_1 u)^2+(J_2 u)^2)\, {\rm d}v_{\hn} \, -\,  V_{\lambda}(d) \int_{\Omega} ((J_1 u)^2+(J_2 u)^2)\,
{\rm d}v_{\hn}\,.
\end{equation*}
Concerning the estimate of $R_2$, using the Poincar\'e inequality (without remainder terms), the estimate $V_{\lambda,y_1, y_2 }(x)\leq 2 V_{\lambda}(d)$ in $\hn \setminus  \overline{\Omega}$, and Lemma \ref{key_lemma},  we have
 \begin{align*}
R_2 & \geq \int_{\hn}  |\nabla_{\hn} (J_{3} u)|^2 \, {\rm d}v_{\hn} -  \int_{\hn \setminus  \overline{\Omega}}  V_{\lambda,y_1, y_2 }(x)\, (J_3 u)^2 \, {\rm d}v_{\hn}\\
& - \sum_{k = 1}^{2} \int_{B(y_k, d)} \left( \frac{\pi^2}{d^2}+2K_{N}(\lambda,d) \right) \, u^2 \, {\rm d}v_{\hn}
\\
&\geq \lambda \int_{\hn}   (J_{3} u)^2 \, {\rm d}v_{\hn} - 2V_{\lambda}(d) \int_{\hn \setminus  \overline{\Omega}} (J_3 u)^2\, {\rm d}v_{\hn} \\
& - \sum_{k = 1}^{2} \int_{B(y_k, d)} \left( \frac{\pi^2}{d^2}+2K_{N}(\lambda,d) \right) \, u^2 \, {\rm d}v_{\hn}\\
&\geq \lambda \int_{\hn}   (J_{3} u)^2 \, {\rm d}v_{\hn} -\left( \frac{\pi^2}{d^2}+2 K_{N}(\lambda,d) \right)\int_{\hn } u^2\, {\rm d}v_{\hn} \,,\end{align*}
where in the last step we exploit the fact that $(J_3 u)^2=u^2$ in $\hn \setminus  \overline{\Omega}$ and that $K_{N}(\lambda,d)>V_{\lambda}(d)$ for all $d>0$ and $N-2 \leq \lambda\leq \left(\frac{N-1}{2}\right)^2$. Summarising, we get
\begin{align*}
 \langle Lu, u \rangle &\geq \lambda \int_{\Omega} ((J_1 u)^2+(J_2 u)^2)\, {\rm d}v_{\hn}+\lambda \int_{\hn}   (J_{3} u)^2 \, {\rm d}v_{\hn} \\
 & -\,  V_{\lambda}(d) \int_{\Omega} ((J_1 u)^2+(J_2 u)^2)  \, {\rm d}v_{\hn}-\left( \frac{\pi^2}{d^2}+2K_{N}(\lambda,d) \right)\int_{\hn } u^2\, {\rm d}v_{\hn} \\
 &\geq \lambda \int_{\hn} u^2\, {\rm d}v_{\hn}-\,  K_{N}(\lambda,d) \int_{\hn} ((J_1 u)^2+(J_2 u)^2+(J_3 u)^2)  \, {\rm d}v_{\hn}\\
 &-\left( \frac{\pi^2}{d^2}+2K_{N}(\lambda,d) \right)\int_{\hn } u^2\, {\rm d}v_{\hn} \\
&\geq\lambda \int_{\hn} u^2\, {\rm d}v_{\hn}-\left( \frac{\pi^2}{d^2}+3K_{N}(\lambda,d)\right)\int_{\hn } u^2\, {\rm d}v_{\hn} \,.
\end{align*}

\end{proof}

 \medskip

 \subsection{Proof of Theorem~\ref{main_thm}}
The proof relies on Proposition~\ref{sym2_thm} for the case of two singularities. Let $\{ J_k \}_{k = 1}^{M + 1}$ be a partition of unity satisfying condition \eqref{parti_assump} and such that $J_k (x) : = J \left( \frac{{ \rm d}(x, y_k)}{d} \right)$ for all $x\in \hn$ and $k=1,...,M$, with $J$ as defined in the proof of Lemma \ref{key_lemma}. Hence, ${\rm Int}({\rm supp}(J_k))=B(y_k,d)$ for $k=1,...,M$. From Lemma \ref{lem2} we know that
\begin{align*}
\langle Lu, u \rangle   := \int_{\hn} |\nabla_{\hn} u|^2 \, {\rm d}v_{\hn} -  \int_{\hn}  V_{\lambda,y_1, \ldots y_M }(x)\,u^2 \, {\rm d}v_{\hn} = \sum_{k = 1}^{M} \langle L(J_k u), (J_k u) \rangle + R_M\,,
\end{align*}
with
\begin{align*}
R_M & = \int_{\hn}  |\nabla_{\hn} (J_{M + 1} u)|^2 -  \int_{\hn \setminus  \Omega}  V_{\lambda,y_1, \ldots y_M }(x) \,u^2 \, {\rm d}v_{\hn}\\
& - \sum_{k = 1}^{M} \int_{B(y_k,d)} \left( \frac{|\nabla_{\hn} J_k|^2}{1 - J_k^2(x)} +  (1 - J_k^2(x))  V_{\lambda,y_1, \ldots y_M }(x) \right) \, u^2 \, {\rm d}v_{\hn}\,,
\end{align*}
where $\Omega:= \cup_{k = 1}^{M} \overline{B(y_k,d)}$. Next, using ${\rm d}(x, y_l) > d$ for all  $x \in B(y_k,d)$ and all $l \neq k,$ we estimate
 $$V_{\lambda,y_1, \ldots y_M }(x)\leq V_{\lambda,y_k }(x)+(M-1)K_{N}(\lambda,d)\quad \text{ for all }x\in B(y_k,d)$$
and, exploiting  in each $B(y_k, d)$ inequality \eqref{unip} with $x_0=y_k$, for all $N-2 \leq \lambda\leq \left(\frac{N-1}{2}\right)^2$ we obtain
\begin{equation}\label{mod_Lk_estimate}\begin{aligned}
&\sum_{k= 1}^M \langle L(J_k u), (J_k u) \rangle\\ &= \sum_{k= 1}^M\left( \int_{B(y_k,d)} |\nabla_{\hn} J_k u|^2 \, {\rm d}v_{\hn} -  \int_{B(y_k,d)}  V_{\lambda,y_1, \ldots y_M }(x) (J_k u)^2 \, {\rm d}v_{\hn} \right)\\
&\geq \lambda \int_{\Omega} \sum_{k = 1}^{M} \, (J_k u)^2\, {\rm d}v_{\hn} \, -\,  (M-1)K_{N}(\lambda,d) \int_{\Omega} \sum_{k = 1}^{M} \, (J_k u)^2\,
{\rm d}v_{\hn} \,.
\end{aligned} \end{equation}
 Concerning the estimate of $R_M$,  we exploit the Poincar\'e inequality (without remainder terms) and the fact ${\rm d}(x, y_k) \geq d$ for all  $x \in \hn \setminus \Omega$ to deduce  that
\begin{align*}
R_M & \geq \lambda \int_{\hn}  (J_{M + 1} u)^2 - M V_{\lambda}(d)  \int_{\hn \setminus  \Omega} \, u^2 \, {\rm d}v_{\hn}\\
& - \sum_{k = 1}^{M} \int_{B(y_k,d)} \left( \frac{|\nabla_{\hn} J_k|^2}{1 - J_k^2(x)} +  (1 - J_k^2(x))  V_{\lambda,y_1, \ldots y_M }(x) \right) \, u^2 \, {\rm d}v_{\hn}\,,
\end{align*}
with $V_{\lambda}(d)$ as given in \eqref{Vdd}. Consider the last quantity in the r.h.s. above we notice that, since $J_l(x)=0$ for $x\in B(y_k,d)$ and $l \neq k,$ we can write
\begin{align*}
&-\int_{B(y_k,d)} \left( \frac{|\nabla_{\hn} J_k|^2}{1 - J_k^2(x)} +  (1 - J_k^2(x))  V_{\lambda,y_1, \ldots y_M }(x) \right) \, |u|^2 \, {\rm d}v_{\hn}\\
=&-\int_{B(y_k,d)} \left( \frac{|\nabla_{\hn} J_k|^2}{1 - J_k^2(x)} +\frac{|\nabla_{\hn} J_l|^2}{1 - J_l^2(x)} +  (1 - J_k^2(x)-J_l^2(x))  V_{\lambda,y_1, \ldots y_M }(x) \right) \, u^2 \, {\rm d}v_{\hn}\\
\geq& -\int_{B(y_k,d)} \left( \frac{|\nabla_{\hn} J_k|^2}{1 - J_k^2(x)} +\frac{|\nabla_{\hn} J_l|^2}{1 - J_l^2(x)}+  (1 - J_k^2(x)-J_l^2(x))  V_{\lambda,y_k,y_l }(x) \right) \, u^2 \, {\rm d}v_{\hn} \\
&-(M-2) V_{\lambda}(d) \int_{B(y_k,d)}  (1 - J_k^2(x))  \, u^2 \, {\rm d}v_{\hn}\,.
\end{align*}
Then we apply Lemma~\ref{key_lemma} on $B(y_k,d)$ with $(y_1, y_2)=(y_k, y_l)$, and we get
\begin{align*}
&\int_{B(y_k,d)} \left( \frac{|\nabla_{\hn} J_k|^2}{1 - J_k^2(x)} +  (1 - J_k^2(x))  V_{\lambda,y_1, \ldots y_M }(x) \right) \, u^2 \, {\rm d}v_{\hn}\\
\geq &  -\int_{B(y_k,d)} \left( \frac{\pi^2}{d^2}+2K_{N}(\lambda,d) \right) \, u^2 \, {\rm d}v_{\hn}-(M-2) V_{\lambda}(d) \int_{B(y_k,d)}  (1 - J_k^2(x))  \, u^2 \, {\rm d}v_{\hn}\,.
\end{align*}
Coming back to the estimate of $R_M$, since $K_{N}(\lambda,d)>V_{\lambda}(d)$ for all $d>0$ and $N-2\leq \lambda\leq  \left(\frac{N-1}{2}\right)^2$, we infer
\begin{align*}
R_M & \geq \lambda \int_{\hn}  (J_{M + 1}u)^2 - M V_{\lambda}(d)  \int_{\hn \setminus  \Omega}  u^2 \, {\rm d}v_{\hn} - \sum_{k = 1}^{M} \int_{B(y_k,d)} \left( \frac{\pi^2}{d^2}+2K_{N}(\lambda,d) \right) \, u^2 \, {\rm d}v_{\hn}\\
&-(M-2) V_{\lambda}(d) \sum_{k = 1}^{M} \int_{B(y_k,d)}  (1 - J_k^2(x))  \, u^2 \, {\rm d}v_{\hn} \\
& \geq \lambda \int_{\hn}  (J_{M + 1} u)^2 - M K_{N}(\lambda,d)  \int_{\hn \setminus  \Omega}  u^2 \, {\rm d}v_{\hn}\\
&  - \int_{\hn}  \frac{\pi^2}{d^2}\, u^2 \, {\rm d}v_{\hn}- K_{N}(\lambda,d) \sum_{k = 1}^{M} \int_{B(y_k,d)}  (M- (M-2)J_k^2(x))  \, u^2 \, {\rm d}v_{\hn} \,.
\end{align*}
The latter inequality, together with \eqref{mod_Lk_estimate} and summing up all terms, yields
\begin{align*}
\langle Lu, u \rangle  & \geq   \lambda \int_{\hn} \sum_{k = 1}^{M+1} \, (J_ku)^2\, {\rm d}v_{\hn} -   \int_{\hn}  \frac{\pi^2}{d^2}\, u^2 \, {\rm d}v_{\hn}\\
& -  K_{N}(\lambda,d) \sum_{k = 1}^{M} \int_{B(y_k,d)} (M+J_k^2(x))  \, u^2 - M  K_{N}(\lambda,d)  \int_{\hn \setminus  \Omega}  u^2 \, {\rm d}v_{\hn}\\
& \geq   \lambda \int_{\hn}  \, u^2\, {\rm d}v_{\hn} - \int_{\hn}  \frac{\pi^2}{d^2}\, u^2 \, {\rm d}v_{\hn}\\
& -  K_{N}(\lambda,d) (M+1)\sum_{k = 1}^{M} \int_{\Omega_k}  \, u^2 - M  K_{N}(\lambda,d)  \int_{\hn \setminus  \Omega}  u^2 \, {\rm d}v_{\hn}\\
& \geq   \lambda \int_{\hn}  \, u^2\, {\rm d}v_{\hn} - \int_{\hn}  \frac{\pi^2}{d^2}\, u^2 \, {\rm d}v_{\hn}-  K_{N}(\lambda,d) (M+1)\, \int_{\hn} \, u^2 \, {\rm d}v_{\hn}\,,
\end{align*}
where we have also exploited $J^2_k(x) \leq 1$ for all $x\in \hn$. The above estimate completes the proof.

 \section{General manifolds}\label{general}
Let $(\mathcal{M},g)$ be an $N\geq 3$ dimensional simply connected Riemannian manifold and ${\rm d}(., .)$ be the geodesic distance on $\mathcal{M}$. Furthermore, denote $K_{x_0}(x)$ the sectional curvature on $\mathcal{M}$ at a point $x_0\in \mathcal{M}$ and direction $x\in \mathcal{M}$. Within this Section we will require that
\begin{equation}\label{-c}
 K_{x_{0}}(x)\leq -c \quad \forall x_0,x\in \mathcal{M} \text{ and for some } c>0\,.
\end{equation}
Hence, only Cartan-Hadamard manifolds will be considered. We comment that this assumption can be replaced by a weaker but more cumbersome one, namely by requiring that sectional curvature in each radial direction with respect to the $M$ poles satisfies the same bound.
\medskip \par
Next, for $\lambda \leq  \left(\frac{N-1}{2} \right)^2$, $c>0$ and $x_0\in \hn$, we set
\begin{align}\label{M_pot}
V_{\lambda,c,x_0}(x) : = H_N(\lambda) \frac{1}{{\rm d^2}(x, x_0)} \,+ \, C_N(\lambda)
\frac{c}{\sinh^2 (\sqrt{c}{\rm d}(x, x_0))} + \, D_N(\lambda) \, g_c({\rm d}(x, x_0)) \,,
\end{align}
where $g_c(r):= \frac{r\,\sqrt{c}\,\coth (\sqrt{c} r)-1}{r^2}$ for all $r>0$ and $H_N(\lambda), C_N(\lambda),D_N(\lambda)$ are as given in \eqref{constants}. Finally, we recall the counterpart of Proposition \ref{improved-hardy} on $\mathcal{M}$:

\begin{prop}\cite[Corollary 4.2]{BGGP}\label{general manifold cor}
Let $(\mathcal{M},g)$ be an $N\geq 3$ dimensional Cartan-Hadamard manifold satisfying condition \eqref{-c} and let $x_0\in \mathcal{M}$. For all $ N-2\leq \lambda \leq  \left(\frac{N-1}{2} \right)^2$ the following inequality holds
 \begin{equation*}
\int_{\mathcal{M}}|\nabla_{g} u|^2\, \emph{d}v_{g} -\lambda \,c\, \int_{\mathcal{M}}  u^2\, \emph{d}v_{g}\\
  \geq  \int_{\mathcal{M}} V_{\lambda,c,x_0}(x) \,u^2 \ \emph{d}v_{g} \,,
  \end{equation*}
for all $u \in C_c^{\infty}(\mathcal{M}),$ where $V_{\lambda,c,x_0}$ is as given in \eqref{M_pot}.
\end{prop}

By exploiting Proposition \ref{improved-hardy} and proceedings along the lines given in the special case $\mathcal{M}=\mathbb{H}^N$ one proves the following Theorem:

 \begin{thm}\label{multip_general}
Let $(\mathcal{M},g)$ be an $N\geq 3$ dimensional Cartan-Hadamard manifold satisfying condition \eqref{-c}. Furthermore, let $y_1, \ldots y_M \in \mathcal{M}$ and $d>0$ as defined in \eqref{d}.
 Then, for all $N-2\leq  \lambda \leq  \left(\frac{N-1}{2} \right)^2$
 the following inequality holds
 \begin{equation*}
\int_{ \mathcal{M}} |\nabla_{g} u|^2 \, {\rm d}v_{g}  +\left( \frac{\pi^2}{d^2}+(M+1)K_{N,c}(\lambda,d) -\lambda c\right)\int_{ \mathcal{M} } u^2\, {\rm d}v_{g} \geq  \int_{ \mathcal{M}}  V_{\lambda,c,y_1, \ldots y_M }(x)\, u^2 \, {\rm d}v_{g}
  \end{equation*}
 for all $u \in C_{c}^{\infty} ( \mathcal{M})$ with
 \begin{equation*}
K_{N,c}(\lambda,d) : = H_N(\lambda) \frac{1}{ d^2} \,+ \, C_N(\lambda)
\frac{c}{\sinh^2 (\sqrt{c}d/2)} + \, D_N(\lambda) \, g_c(d/2) \,,
\end{equation*}
 with $g_c$ as given in \eqref{M_pot} and, $V_{\lambda,c,y_k}(x)$ being as in \eqref{M_pot},
 $$V_{\lambda,c,y_1, \ldots y_M }(x) = \sum_{i=1}^M V_{\lambda,c,y_k}(x)\,.$$
\end{thm}

Note that for $\lambda=N-2$,  $H_N(\lambda)=\frac{(N-2)^2}{4}$ and $C_N(\lambda)=0$. Hence, since $g_c(r) \rightarrow 0$ for every $r>0$ as $c\rightarrow 0^+$, one recovers the Euclidean inequality \eqref{m_euc_hardy}.

\medskip
From Theorem \ref{multip_general} one may derive the counterpart of Corollary \ref{hardy_thm} on $\mathcal{M}$. Let $\bar d_c=\bar d_c(M,N)>0$ be defined implicitly as follows
  \begin{equation}\label{overd_manif}
  \frac{\pi^2 }{\bar d_c^2}+(M+1)\left(\frac{(N-2)^2}{4} \frac{1}{\bar d_c^2}+ \frac{(N-3)(N-2)}{2} g_c(\bar d_c/2)\right)=(N-2)c\,.
   \end{equation}
 We have
\begin{cor}\label{hardy_thm_manif}
Let $(\mathcal{M},g)$ be an $N\geq 3$ dimensional Cartan-Hadamard manifold satisfying condition \eqref{-c}. Furthermore, let $y_1, \ldots y_M \in \mathcal{M}$, $d>0$ as defined in \eqref{d} and $\bar d_c=\bar d_c(M,N)>0$ as defined in \eqref{overd_manif}.  If $d\geq \bar d_c $, the following inequality holds

 \begin{equation}\label{hardy_multi_manif} \begin{aligned}
 \int_{\mathcal{M}} |\nabla_{g} u|^2 \, {\rm d}v_{g}
 &  \geq    \frac{(N-2)^2}{4}  \sum_{k = 1}^M \int_{\mathcal{M}}  \frac{u^2}{{\rm d^2}(x, y_k)} \, {\rm d}v_{g}\\
 &+ \frac{(N-2)(N-3)}{2}  \sum_{k = 1}^M  \int_{\mathcal{M}} g_c({\rm d}(x, y_k))\,u^2 \, {\rm d}v_{g}
 \end{aligned} \end{equation}
 for all $u \in C_{c}^{\infty} ( \mathcal{M})$, where $g_c$ is as given in \eqref{M_pot}.
 \end{cor}

Inequality \eqref{hardy_multi_manif} should be compared with inequality (3.5) in \cite{FFK} which, under the same assumptions of Corollary \ref{hardy_thm_manif}, reads:
 \begin{equation} \begin{aligned}\label{hardy_multi_crista}
 \int_{\mathcal{M}} |\nabla_{g} u|^2 \, {\rm d}v_{g}
 &  \geq    \frac{(N-2)^2}{M^2}  \sum_{1 \leq i< j \leq M} \int_{\mathcal{M}} \left| \frac{\nabla_g ({\rm d}(x, y_i)) }{{\rm d}(x, y_i)}-\frac{\nabla_g ({\rm d}(x, y_j)) }{{\rm d}(x, y_j)}\right|^2 u^2 \, {\rm d}v_{g}\\
 &+ \frac{(N-2)(N-1)}{M}  \sum_{k = 1}^M  \int_{\mathcal{M}} g_c({\rm d}(x, y_k))\,u^2 \, {\rm d}v_{g}
 \end{aligned} \end{equation}
 for all $u \in C_{c}^{\infty} (\mathcal{M})$. Hence, the potential in \eqref{hardy_multi_manif}, compared to that in \eqref{hardy_multi_crista}, is smaller near poles for $M=2$ but larger for $M\geq 3$.
\par \medskip
We conclude with the analogue of Corollary \ref{hardy_thm_lam} on $\mathcal{M}$, which follows along similar lines. For $N-2\leq  \lambda < \left(\frac{N-1}{2} \right)^2$, let $\bar d_{\lambda,c}=\bar d_{\lambda,c}(M,N)>0$ be defined implicitly as follows
  \begin{equation*}
  \frac{\pi^2 }{\bar d_{\lambda,c} }+(M+1)\left(\frac{1}{4} \frac{1}{\bar d_{\lambda,c}^2}+\frac{(N-1)(N-3)}{4}\ \frac{c}{\sinh^2 ((\sqrt{c}\bar d_{\lambda,c})/2)} \right)=\left(\left(\frac{N-1}{2} \right)^2 -\lambda \right)c\,.
   \end{equation*}
 There holds:
\begin{cor}\label{hardy_thm_lamc}
Let $(\mathcal{M},g)$ be an $N\geq 3$ dimensional Cartan-Hadamard manifold satisfying condition \eqref{-c}. Furthermore, let $y_1, \ldots y_M \in \mathcal{M}$ and $d>0$ as defined in \eqref{d}. Let $N-2\leq \lambda < \left(\frac{N-1}{2} \right)^2$, if $d\geq \bar d_{\lambda,c} $, the following inequality holds

 \begin{equation*}
\begin{aligned} \int_{\mathcal{M}} |\nabla_{g} u|^2 \, {\rm d}v_{g}
 &  \geq  \lambda c   \int_{\mathcal{M}}  u^2 \, {\rm d}v_{g} +   \frac{1}{4}  \sum_{k = 1}^M  \int_{\mathcal{M}}  \frac{u^2}{{\rm d^2}(x, y_k)} \, {\rm d}v_{g}\\ \notag
 &+ \frac{(N-1)(N-3) c}{4}  \sum_{k = 1}^M  \int_{\mathcal{M}}    \frac{u^2}{ \sinh^2 ((\sqrt{c}{\rm d}(x, y_k))/2)}\, \, {\rm d}v_{g}
\end{aligned}
 \end{equation*}
 for all $u \in C_{c}^{\infty} (\mathcal{M})$.
 \end{cor}

\subsection{Proof of Theorem \ref{multip_general}}
The proof of Theorem \ref{multip_general} follows by repeating the proof of Theorem~\ref{main_thm} with very few changes, hence we omit it. We limit ourselves to remark that the counterpart of Lemma \ref{key_lemma} on $\mathcal{M}$ follows by replacing Proposition \ref{improved-hardy} with Proposition \ref{general manifold cor} and by noting that for every $x_0 \in \mathcal{M}$ the eikonal condition holds:
$$|\nabla_{g} { \rm d}(., x_0)|^2=1 \quad \text{a.e. on } \mathcal{M}\,.$$

\section{Remarks on the criticality issue}\label{super_solution_sec}
The criticality issue of Proposition \ref{improved-hardy} does not carry over to the multipolar framework of Theorem \ref{main_thm}. A critical multipolar Hardy-weight can instead be derived from the general results of \cite{DFP}. For simplicity we focus on the hyperbolic setting. Let $y_1, \ldots y_M \in \hn$ and let $G(x, y_i)$ be the Green function of $-\Delta_{\hn}$ at $y_i$, which reads
 $$
 G(x, y_i) = \int_{{\rm d}(x, y_i)}^{\infty} (\sinh s)^{-(N-1)} \, {\rm d}s\,.
 $$
Finally, we set
\begin{equation}\label{W}
W(x):= \left(\frac{1}{M + 1} \right)^2 \left[ \sum_{i = 1}^{M} \left| \dfrac{|\nabla  G(x, y_i) |}{ G(x, y_i)}\right|^2
  + \sum_{1 \leq i < j\leq M}
   \left|  \dfrac{|\nabla  G(x, y_i)|}{G(x, y_i)} -  \dfrac{|\nabla  G(x, y_j) |}{G(x, y_j)}\right|^2 \right]
\end{equation}
and we state the counterpart of inequality \eqref{m_euc_hardy_4} in $\hn$:
\begin{prop}\label{multicri}
Let $N \geq 3$ and $M \geq 2$. Furthermore, let $y_1, \ldots y_M \in \hn$, $d>0$ as defined in \eqref{d} and $W(x)$ as defined in \eqref{W}. Then, the following multipolar inequality holds
 \begin{equation*}\label{hardy_multi_critic}
 \int_{\hn} |\nabla_{\hn} u|^2 \, {\rm d}v_{\hn}
  \geq   \int_{\hn} W(x)\, u^2  \, {\rm d}v_{\hn}
 \end{equation*}
 for all $u \in C_{c}^{\infty} (\hn)$, and the operator $-\Delta_{\hn}-W$ is critical in $\hn \setminus \{ y_1,...,y_M \}$.
 \end{prop}

It is worth noting that
 $$
 \dfrac{|\nabla  G(x, y_i) |}{ G(x, y_i)} \sim
\, \frac{N-2}{{\rm d}(x, y_i)} \quad \text{as } x \rightarrow y_i\,,
 $$
 while
 $$
\dfrac{|\nabla  G(x, y_i) |}{ G(x, y_i)}= N-1+\frac{(N-1)^2}{2(N-2)} \, (\sinh({\rm d}(x, y_i)))^{ -2} \quad \text{as } x \rightarrow +\infty\,.
$$
Namely,
$$W(x) \sim M\,\left(\frac{N-2}{M + 1} \right)^2\,\frac{1}{{\rm d}^2(x, y_i)} \quad \text{as } x \rightarrow y_i \quad \text{and} \quad W(x) \sim \frac{M(N-1)^2}{(M + 1)^2} \quad \text{as } x\rightarrow +\infty\,. $$
Hence, the above weight is smaller than that in \eqref{hardy_multi_crista} (obtained in \cite{FFK}) near poles but larger at infinity. Furthermore, near poles one recovers the same behavior of the Euclidean weight in \eqref{m_euc_hardy_4}.

\begin{proof}
The proof of Proposition \ref{multicri} follows from Theorem 12.3 in \cite{DFP} which in $\hn$ reads: let $P$ be a subcritical operator, if there exist $u_0,...,u_{M}$, $M+1$ positive solutions to the equation $P u=0$ in $\hn \setminus \{ y_1,...,y_M \}$ of minimal growth near $y_i$ and $\infty$, i.e.
$$\lim_{x \rightarrow \infty} \frac{u_i(x)}{u_0(x)}=0 \text{ for all }  i\neq 0 \quad \text{and} \quad \lim_{x \rightarrow x_i} \frac{u_j(x)}{u_i(x)}=0 \text{ for all } j\neq i\,,$$
then
$$W_{\alpha}(x)=\sum_{i<j} \alpha_i \alpha_j \left|\nabla\left(\log \frac{u_i(x)}{u_j(x)} \right) \right|^2 \quad \text{with }\alpha_j\in (0,1/2] \text{ and }\sum_{j=0}^M \alpha_j=1\,,$$
is a critical Hardy-weight for $P$. \par In our case, we have $P=-\Delta_{\hn}$ and we may take the functions $u_0(x) = 1$ and $u_i(x) = G(x, y_i)$, with $i=1,...,M$, which are positive solutions to $-\Delta_{\hn}u=0$ in $\hn \setminus \{ y_1,...,y_M \}$ and satisfy the required growth conditions. Finally, the weight $W$ in \eqref{W} is obtained by choosing in $W_{\alpha}$ uniformly  $\alpha_j = \frac{1}{M + 1}$ for all $j = 0, \ldots, M$.
\end{proof}


\par\bigskip\noindent
\textbf{Acknowledgments.}  The first author is partially supported by MIUR grant Dipartimenti di Eccellenza 2018-2022. The second author is partially supported by an INSPIRE faculty fellowship. The third author is partially supported by the PRIN project ``Equazioni alle derivate parziali di tipo ellittico e parabolico: aspetti geometrici, disuguaglianze collegate, e applicazioni'' (Italy). The first  and third authors are members of the Gruppo Nazionale per l'Analisi Matematica, la Probabilit\`a e le loro Applicazioni (GNAMPA) of the Istituto Nazionale di Alta Matematica (INdAM).

\normalsize


\end{document}